\documentclass[12pt,reqno]{amsart}
\usepackage{amsmath, amssymb, amsthm, amsfonts} 
\usepackage{url}
\usepackage[breaklinks]{hyperref}
\usepackage{tikz}
\usetikzlibrary{decorations.markings}
\usepackage{array}
\newcommand{\Arg}{\operatorname{Arg}}

\renewcommand{\Re}{\operatorname{Re}}

\renewcommand{\(}{\left\(}
\renewcommand{\)}{\right\)}
\renewcommand{\[}{\left\[}
\renewcommand{\]}{\right\]}

\numberwithin{equation}{section}
 \theoremstyle{plain}
\newtheorem{theorem}{Theorem}[section]
\newtheorem{lemma}[theorem]{Lemma}
\newtheorem{remark}[]{Remark}

\newtheorem{corollary}[theorem]{Corollary}

   \makeatletter
\def\proof{\@ifnextchar[{\@oproof}{\@nproof}}
\def\@oproof[#1][#2]{\trivlist\item[\hskip\labelsep\textit{#2 Proof of\
#1.}~]\ignorespaces}
\def\@nproof{\trivlist\item[\hskip\labelsep\textit{Proof.}~]\ignorespaces}

\makeatother

\begin{document}
\title[A number field extension of Ramanujan's identity]{A number field extension of Ramanujan's identity and its integral representation via the generalized Riemann $\Xi$-function}

\author{Diksha Rani Bansal}
\address{Diksha Rani\\ Department of Mathematics \\
Indian Institute of Technology Indore \\
Indore, Simrol, Madhya Pradesh 453552, India.}  
\email{dikshaba1233@gmail.com,  mscphd2207141001@iiti.ac.in}


\begin{abstract}
On pages 219--220 of Ramanujan's Lost Notebook appears an intriguing modular transformation whose integral representation is expressed in terms of the Riemann $\Xi$-function. In this paper, we establish a number field analogue of this identity. We also obtain an integral representation for the transformation formula associated with the functional equation of the higher power of the Dedekind zeta function recently studied by the present author and Maji. Interestingly, we recover Hardy's integral representation of the well-known Jacobi theta relation involving the Riemann $\Xi$-function. We further obtain a number-field analogue of the extended Ramanujan--Koshliakov formula established by Dixit.
\end{abstract}

\maketitle

\section{{Introduction}}
Ramanujan’s Lost Notebook \cite{Ramanujan} remains a rich source of results in analytic number theory. One of the notable results on pages 219–220, is a transformation formula involving the logarithmic derivative of the gamma function, the Riemann zeta function and the function
\begin{align}\label{Series Lambda}
\lambda(x) := \psi(x) + \frac{1}{2x} - \log(x),
\end{align}
where $\psi(x) = \frac{\Gamma'(x)}{\Gamma(x)}$. If $\alpha$, $\beta$ are positive numbers such that $\alpha\beta = 1$, then the identity is given by
\begin{align}\label{Id 220}
G(\alpha) &= G(\beta) = -\frac{1}{\pi^{\frac{3}{2}}}\int\limits_0^\infty \left|\Xi\left(\frac{t}{2}\right)\Gamma\left(\frac{-1+it}{4}\right)\right|^2\frac{\cos\left(\frac{t}{2}\log\alpha\right)}{(1+t^2)}\,\text{d}t,
\end{align}
where $G(x)$ is defined by
\begin{align*}
G(x) = \sqrt{x}\left\{\frac{\gamma - \log(2\pi x)}{2x} + \sum_{m=1}^{\infty}\lambda(mx)\right\}.
\end{align*}
What makes this identity especially striking is that its integral representation involves the Riemann $\Xi$-function defined as
$$
\Xi(t) = \left(-\frac{1}{8} - \frac{t^2}{2} \right)\pi^{-\frac{1}{4}-\frac{it}{2}}\Gamma\left(\frac{1}{4} + \frac{it}{2}\right)\Gamma\left(\frac{1}{2} + it\right)\zeta\left(\frac{1}{2} + it\right),
$$
whose real zeros give the non-trivial zeros of the Riemann zeta function on the critical line. The result \eqref{Id 220} can be viewed as an example of a general type of relation where a function satisfies $F(\alpha)=F(\beta)$ whenever $\alpha\beta=1$. Similar identities later appeared in the work of Koshliakov \cite{Koshliakov} and Ferrar \cite{Ferrar1936}.

The identity \eqref{Id 220} has subsequently been studied from several points of view. Its first equality was rediscovered by Guinand \cite[p.~18]{Guinand47} in a slightly different form. He later gave a simpler proof with extensions involving derivatives of the $\psi$-function in \cite{Guinand52}; however, he did not note its connection with Ramanujan’s integral involving the Riemann's function. Berndt and Dixit \cite{BD2010} subsequently gave a proof of \eqref{Id 220} using integrals studied by Ramanujan \cite[Equation (22), p.~260]{Ramanujan1915}; see also \cite[Equation (22), p.~17]{Ramanujan1962}. Dixit \cite{Dixit2010} obtained another proof using Mellin transforms and also derived the identity as a limiting case of a more general transformation formula \cite[Theorem 1.4]{Dixit2011}. Character analogues of Ramanujan type integrals involving the Riemann $\Xi$-function can be found in \cite{Dixit2012}. 
Readers may also see \cite{Conway78, DK2021}.
%
%
The transformation phenomenon $F(\alpha)=F(\beta)$ where $\alpha, \beta > 0$ such that $\alpha\beta=1$, however, occurs in a much broader setting. One of its fundamental examples is the Jacobi theta transformation
\begin{align*}
\mathcal{H}_1\left(\alpha \right) = \mathcal{H}_1(\beta),
\end{align*}
where $\mathcal{H}_1(x)$ is the well-known Jacobi theta function, defined as 
\begin{align*}
\mathcal{H}_1(x) = \sqrt{x}\left(\frac{1}{x} - 2\sum_{n = 1}^{\infty}e^{-\pi n^2 x^2}\right),
\end{align*}
which is equivalent to the functional equation of the Riemann zeta function. Hardy \cite{Hardy} established the following integral representation of the Jacobi theta relation to prove the infinitude of the non-trivial zeros of the Riemann zeta function on the critical line,
\begin{align}\label{Jacobi relation}
\mathcal{H}_1\left(\alpha \right) = \mathcal{H}_1(\beta) = \frac{4}{\pi}\int\limits_{0}^{\infty}\frac{\Xi\left(\frac{t}{2}\right)}{\left(1 + t^2\right)} \cos\left(\frac{t}{2}\log\alpha\right)\mathrm{d}t.
\end{align} 
Another important example is the Ramanujan--Koshliakov \cite[p.~97]{GB}, \cite{Koshliakov}, \cite[p.~253]{Ramanujan},  formula
\begin{align*}
\mathcal{H}_2\left(\alpha \right) = \mathcal{H}_2(\beta),
\end{align*}
which is equivalent to the functional equation of $\zeta^2(s)$, where $\mathcal{H}_2(x)$ is given by
$$\mathcal{H}_2(x) = \sqrt{x}\left(\frac{\gamma - \log(4\pi x)}{x} - 4\sum_{n=1}^{\infty}d(n)\text{K}_0(2n\pi x) \right)$$
Here, $\gamma$ is the Euler-Mascheroni constant, $d(n)$ denotes the number of divisors of $n$, and $K_0$ is the modified Bessel function of the second kind.
Dixit \cite{Dixit2010} developed this perspective by expressing the common values of several transformations of the form
$F(\alpha)=F(\beta)$ in terms of integrals involving the Riemann $\Xi$-function. In particular, he obtained the following extended form of the Ramanujan--Koshliakov formula. If $\alpha,\beta>0$ satisfy $\alpha\beta=1$, then
\begin{align}
\mathcal{H}_2\left(\alpha \right) = \mathcal{H}_2(\beta) = -\frac{32}{\pi}\int\limits_{0}^{\infty}\frac{\Xi^{2}\left(\frac{t}{2}\right)\cos\left(\frac{t}{2}\log\alpha\right)}{\left(1 + t^2\right)^{2}} \mathrm{d}t. \label{Extended Koshliakov formula}
\end{align}
Ferrar \cite{Ferrar1936} obtained a simultaneous extension of the
Jacobi theta transformation and the Ramanujan--Koshliakov formula by
considering a transformation associated with the functional equation
of $\zeta^k(s)$, where $k$ is a positive integer. Recently, the
present author and Maji \cite{BM2} extended Ferrar's work to number
fields, obtaining a transformation associated with the functional
equation of $\zeta_{\mathbb{F}}^m(s), m \in \mathbb{Z}\setminus\{0\}$, where $\zeta_{\mathbb{F}}(s)$ is the Dedekind zeta
function defined below. 

Let $\mathbb{F}$ be a number field of degree $ d=r_1+2r_2, $ where $r_1$ denotes the number of real embeddings of $\mathbb{F}$ and $r_2$ the number of pairs of complex conjugate embeddings. Let $D$ be the absolute value of the discriminant of $\mathbb{F}$, and let $\mathcal{O}_{\mathbb{F}}$ denote its ring of integers. For each positive integer $n$, let $\mathtt{a}_{\mathbb{F}}(n)$ denote the number of non-zero integral ideals of $\mathcal{O}_{\mathbb{F}}$ having norm $n$.
The Dedekind zeta function associated with $\mathbb{F}$, is given by
\begin{align}\label{Dedekind zeta}
\zeta_{\mathbb{F}}(s)
=\sum_{n=1}^{\infty}\frac{\mathtt{a}_{\mathbb{F}}(n)}{n^s}, \quad \Re(s)>1.
\end{align}
The Dedekind zeta function $\zeta_{\mathbb F}(s)$ satisfies the functional equation,
\begin{align}\label{Dedekind zeta functional equation}
\Omega_{\mathbb{F}}(s) = \Omega_{\mathbb{F}}(1-s),
\end{align}
where 
\begin{align}\label{Omega_F of fnal eqn}
\Omega_{\mathbb{F}}(s) &= \left(\frac{D}{4^{r_2}\pi^d}\right)^{\frac{s}{2}}\Gamma^{r_1}\left(\frac{s}{2}\right)\Gamma^{r_2}(s)\zeta_{\mathbb{F}}(s).
\end{align}
It is known that $\zeta_{\mathbb{F}}(s)$ has a simple pole at $s=1$ with the residue given by the class number formula: 
\begin{align}\label{Class number formula}
\lim_{s \rightarrow 1}(s-1)\zeta_{\mathbb{F}}(s) = \frac{2^{r_1}(2\pi)^{r_2}}{\sqrt{D}}\frac{R_{\mathbb{F}}h_{\mathbb{F}}}{w_{\mathbb{F}} } :=H_{ \mathbb{F}},
\end{align}
where $w_{\mathbb{F}}$ is the number of roots of unity contained in $\mathcal{O}_\mathbb{F}, h_{\mathbb{F}}$ denotes the class number, and $R_{\mathbb{F}}$ is the regulator of $\mathbb{F}$. 
We next introduce the following entire function:
{\allowdisplaybreaks
\begin{align}\label{xi function}
\xi_{\mathbb{F}}(s) &:= \frac{1}{2}s(s-1) \left(\frac{D}{4^{r_2}\pi^d}\right)^{\frac{s}{2}}\Gamma^{r_1}\left(\frac{s}{2}\right)\Gamma^{r_2}(s)\zeta_{\mathbb{F}}(s) \nonumber \\
&= \frac{1}{2}s(s-1)\Omega_{\mathbb{F}}(s) = \xi_{\mathbb{F}}(1-s).
\end{align}}
Let $\Xi_{\mathbb F}(t) = \xi_{\mathbb{F}}\left(\frac{1}{2}+it\right)$ denote the generalized Riemann $\Xi$-function associated with $\zeta_{\mathbb F}(s)$, defined by
\begin{align}\label{Xi_F(t)}
\Xi_{\mathbb F}(t) := \left(-\frac{1}{8} - \frac{t^2}{2} \right)\left(\frac{D}{4^{r_2}\pi^d}\right)^{\frac{1}{4}+\frac{it}{2}}\Gamma^{r_1}\left(\frac{1}{4} + \frac{it}{2}\right)\Gamma^{r_2}\left(\frac{1}{2} + it\right)\zeta_{\mathbb{F}}\left(\frac{1}{2} + it\right).
\end{align}
Using $\overline{\Gamma(s)} = \Gamma(\overline{s})$ and \eqref{xi function}, one can prove that $\Xi_{\mathbb F}(t)$ is real and even for real values of $t$.
%
Let $k$ be a positive integer and $\mathbb{F}$ be any number field of degree $d$. Let $\mathtt{a}_{\mathbb{F}, k}(n)$ be defined by
\begin{align}
\zeta_{\mathbb{F}}^k(s) := \sum_{n=1}^{\infty}\frac{\mathtt{a}_{\mathbb{F}, k}(n)}{n^s}, \quad \Re(s)>1.  \label{Coeff a_k}
\end{align}
The author and Maji \cite[Theorem~2.1]{BM2} proved that the functional equation of $\zeta_\mathbb{F}^k(s)$ is equivalent to the theta relation 
\begin{align}\label{Theta reln for dedekind zeta power k}
W_{\mathbb{F}, k}\left(\frac{1}{x}\right) = \sqrt{x}W_{\mathbb{F}, k}(x), \quad \text{for} \,\,\, x \neq 0,\,\,  |\Arg(x)| < \min\left\{\frac{\pi d}{2}, \pi\right\},
\end{align}
where $W_{\mathbb{F}, k}(x)$ is given by
\begin{align}\label{W_F,k(x)}
W_{\mathbb{F}, k}(x) := \sum_{n=1}^{\infty}\mathtt{a}_{\mathbb{F}, k}(n)\tilde{Z}_{kr_1, kr_2}\left(\frac{2^{kr_2}\pi^{kd/2} n\sqrt{x}}{\sqrt{D^k}} \right) - R_0(x),
\end{align}
with
 $$
 R_0(x) = \frac{1}{(k-1)!}\lim_{s \rightarrow 0} \frac{\mathrm{d}^{k-1}}{\mathrm{d}s^{k-1}}\left(s^k \Omega_{\mathbb{F}}^k(s)x^{-s/2}\right),
 $$
and $\tilde{Z}_{r_1, r_2}(x)$ is defined in \eqref{Defn of Z}. Using the functional equation \eqref{Dedekind zeta functional equation} of $\zeta_{\mathbb F}(s)$ in the definition of $R_0(x)$, one can verify that $R_0(x)=-\frac{1}{\sqrt{x}}R_1(x)$, where 
$$
R_1(x) = \frac{1}{(k-1)!}\lim_{s\to 1}\frac{\text{d}^{k-1}}{\text{d}s^{k-1}}\left((s-1)^{k}\Omega_\mathbb{F}^k(s)x^{s/2} \right).
$$
We substitute this value of $R_0(x)$ into \eqref{W_F,k(x)}, replace $x$ by $x^2$ in \eqref{Theta reln for dedekind zeta power k}, and multiply the resulting equation by $\frac{1}{\sqrt{x}}$. We then rearrange the equation by transferring the series terms to the opposite sides. Setting $\alpha=x$ and $\beta=\frac{1}{x}$, so that $\alpha\beta=1$, we rewrite \eqref{Theta reln for dedekind zeta power k} as
\begin{align}\label{Theta reln for dedekind zeta power k-x by x^2}
\mathcal{H}_{\mathbb{F}, k}(\alpha) = \mathcal{H}_{\mathbb{F}, k}(\beta),
\end{align}
where
\begin{align}
\mathcal{H}_{\mathbb{F},k}(x) &= \sqrt{x}\left\{R^{(k)}_1\left(\frac{1}{x}\right) - \sum_{n=1}^{\infty}\mathtt{a}_{\mathbb{F}, k}(n)\tilde{Z}_{kr_1, kr_2}\left(\frac{2^{kr_2}\pi^{kd/2} nx}{\sqrt{D^k}} \right)\right\}, \label{New W} \\
R^{(k)}_1(x) &= \frac{1}{(k-1)!}\lim_{s\to 1}\frac{\mathrm{d}^{k-1}}{\mathrm{d}s^{k-1}}\left((s-1)^{k}\Omega_\mathbb{F}^k(s)x^{s} \right).
\end{align}
The main goals of this paper are to establish a number field analogue of Ramanujan's identity \eqref{Id 220} and to derive an integral representation of \eqref{Theta reln for dedekind zeta power k-x by x^2} involving the generalized Riemann $\Xi$-function.

Now, we define the function $\tilde{Z}_{r_1, r_2}(x)$ which is a particular case of the Steen function. The Steen function $V(x| a_1, a_2, \ldots, a_n)$ is defined by the Mellin--Barnes integral
\begin{align*}
V(x| a_1,a_2,\ldots,a_n) :=
\frac{1}{2\pi i}\int\limits_{(c)}
\prod_{j=1}^{n}\Gamma(s+a_j)x^{-s} \, \mathrm{d}s, \quad |\Arg(x)|<\frac{\pi n}{2},
\end{align*}
where $(c)$ denotes the vertical contour from $c-i\infty$ to $c+i\infty$. Throughout, $\Arg(x)$ denotes the principal argument, with
$ -\pi<\Arg(x)\leq\pi.$
The contour is chosen so that all the poles of the gamma factors in the integrand lie on the same side of $(c)$.
The function $V(x| a_1, a_2, \ldots, a_n)$ is a particular case of the Meijer $G$-function; see \cite[p.~415, Definition~16.17]{OLBC2010}. For suitable choices of the parameters, it specializes to several familiar special functions like,
\begin{align}
V(x \mid 0) &= e^{-x}, \quad c > 0, \label{Exp fn formula} \\
V(x\mid a,b) &= 2x^{\frac{a+b}{2}}
\mathrm{K}_{a-b}\left(2\sqrt{x}\right),
\quad c>\max\{-a,-b\}, \label{Bessel fn formula}
\end{align}
where $\mathrm{K}_{\nu}$ denotes the modified Bessel function of the second kind of order $\nu$. For further properties of the Steen function, we refer the reader to \cite{KT, Steen}. 
We next consider a function closely related to the Steen function, originally studied by Koshliakov \cite[p.~122]{KoshI} and subsequently employed by Dixit, Gupta, and Vatwani \cite[p.~5, Equation~(2.2)]{DGV2022}.
For $c=\Re(s)>0$, define
\begin{align}\label{Defn of Z}
\tilde{Z}_{r_1, r_2}(x) := \frac{1}{2\pi i}\int\limits_{(c)}\Gamma^{r_1}\left(\frac{s}{2}\right)\Gamma^{r_2}(s)x^{-s} \text{d}s, \quad \text{where} \,\,\, |\Arg(x)| < \frac{\pi d}{4}.
\end{align}
In particular, we have
\begin{align}
\tilde{Z}_{0, r_2}(x) = V(x| \bar{0}_{r_2}), \qquad  
\tilde{Z}_{r_1, 0}(x) = 2V(x^2| \bar{0}_{r_1}), \label{Z value real}
\end{align}
where $\bar{0}_m$ denotes the $m$-tuple $(0,\ldots,0)$. 

The main results of this paper are twofold. First, we establish a number-field analogue of identity \eqref{Id 220}. We remark that the function $\lambda(x)$ in \eqref{Series Lambda} admits a series representation valid only for $|x|<1$, whereas its integral representation \cite[p.~201, formula~5.74]{Oberhettinger},
\begin{align*}
  \lambda(x) = -\frac{1}{2\pi i}\int\limits_{(c)}
  \frac{\pi\,\zeta(1-s)}{\sin(\pi s)}\,x^{-s}\,\mathrm{d}s, \quad c>1,
\end{align*}
extends analytically to $\operatorname{Re}(x)>0$. This
motivates defining $\lambda_{\mathbb{F}}(x)$ in the number-field setting via 
\begin{align}\label{Generalized Lambda}
  \lambda_{\mathbb{F}}(x)
  := \frac{1}{2\pi i}\int\limits_{(1+\delta)}
  \frac{\zeta_\mathbb{F}(1-s)}{\sin^{d}(\pi s)}\,x^{-s}\,\mathrm{d}s,
  \quad 0<\delta<1,
\end{align}
valid for all $\Re(x)>0$. A series expansion of $\lambda_\mathbb{F}(x)$ has been obtained via residues in
Lemma~\ref{Series defn of Lambda_F}.
The absolute convergence of \eqref{Generalized Lambda} follows from
\begin{equation*}
  \left|\frac{\zeta_\mathbb{F}(1-s)}{\sin^{d}(\pi s)}\right|
  \ll |t|^{A}\,e^{-\frac{d\pi|t|}{2}}, \quad |t|\to\infty,
\end{equation*}
where the exponential decay of $|\sin(\pi s)|^{-d}$ dominates the
polynomial growth of $|\zeta_\mathbb{F}(1-s)|$. Moreover, since
$\mathtt{a}_{\mathbb{F}}(m)\ll m^{\varepsilon}, \epsilon > 0$, one has
\begin{equation*}
  \left|\mathtt{a}_{\mathbb{F}}(m)\,\lambda_{\mathbb{F}}(mx)\right|
  \ll m^{\varepsilon-c}\,x^{-c},
  \quad c\in(1,2),\;\varepsilon\in(0,c-1),
\end{equation*}
so the series $\sum_{m=1}^{\infty}\mathtt{a}_{\mathbb{F}}(m)
\lambda_{\mathbb{F}}(mx)$ converges absolutely for
$\operatorname{Re}(x)>0$.

Second, we extend the transformation formula for
$\zeta_{\mathbb{F}}^{k}(s)$, $k\in\mathbb{N}$, expressing its common
value as an integral involving the generalized Riemann $\Xi$-function.
The main results are stated in the next section.

\section{Main Results}
\begin{theorem}\label{Theorem Number field analogue Id 220}
Let $\mathbb{F}$ be a number field of degree $d$, $\mathtt{a}_{\mathbb{F}}(m)$, $\Xi_\mathbb{F}(t)$ and $\lambda_{\mathbb{F}}(x)$ as defined in \eqref{Dedekind zeta}, \eqref{Xi_F(t)} and \eqref{Generalized Lambda}, respectively. For $\Re(x)>0$, define
\begin{align}\label{G(x) for no field}
G_{\mathbb{F}}(x) &= \sqrt{x}\left\{\sum_{m=1}^{\infty}\mathtt{a}_{\mathbb{F}}(m)\lambda_{\mathbb{F}}(mx) - \mathcal{R}^{(r_2+2)}_1\left(\frac{1}{x}\right)\right\},
\end{align}
where
\begin{align*}
\mathcal{R}^{(r_2+2)}_1(x) &= \frac{1}{(r_2+1)!}\lim_{s\to 1}\frac{\mathrm{d}^{r_2+1}}{\mathrm{d}s^{r_2+1}}\left((s-1)^{r_2+2}\frac{\zeta_\mathbb{F}(s)\zeta_\mathbb{F}(1-s)}{\sin^{d}(\pi s)}x^s\right).
\end{align*}
Then, for any positive real numbers $\alpha$ and $\beta$ satisfying $\alpha\beta=1$, we have the transformation formula
\begin{align}\label{Number field analogue of transformation formula}
&G_{\mathbb{F}}(\alpha) = G_{\mathbb{F}}(\beta) \notag \\
&= \frac{2^{5-5r_1-r_2}}{\sqrt{D}\pi^{\frac{3r_1}{2}+r_2+1}}\int\limits_{0}^{\infty}\bigg|\Xi_\mathbb{F}\left(\frac{t}{2}\right)\Gamma^{r_1}\left(\frac{-1+it}{4}\right)\Gamma^{r_2}\left(\frac{-1+it}{2}\right)\bigg|^2\frac{\cos\left(\frac{t}{2}\log\alpha\right)}{(1+t^2)^{2-r_1-r_2}}\,\mathrm{d}t.
\end{align}
\end{theorem}
For $\mathbb{F}=\mathbb{Q}$, the above result recovers Ramanujan's original identity, leading to the following corollary.
\begin{corollary}\label{Recovery of main id}
When $\mathbb{F} = \mathbb{Q}$, Theorem \ref{Theorem Number field analogue Id 220} reduces to the identity \eqref{Id 220}.
\end{corollary}
The following theorem extends the transformation formula \eqref{Theta reln for dedekind zeta power k-x by x^2} by expressing its common value as an integral involving the generalized Riemann $\Xi$-function defined in \eqref{Xi_F(t)}.
\begin{theorem}\label{Theorem Number field analogue Id 2nd paper}
Let $\mathbb{F}$ be a number field of degree $d$ and let $k\geq 1$ be an integer. Let $\Omega_\mathbb{F}(s)$, $\mathtt{a}_{\mathbb{F}, k}(n)$, $\mathcal{H}_{\mathbb{F},k}(x)$ and $\tilde{Z}_{r_1, r_2}(x)$ be defined as in \eqref{Omega_F of fnal eqn}, \eqref{Coeff a_k}, \eqref{New W} and \eqref{Defn of Z}, respectively. For positive real numbers $\alpha$ and $\beta$ such that $\alpha\beta=1$, we have
\begin{align}\label{Extended trans formula}
\mathcal{H}_{\mathbb{F}, k}(\alpha) = \mathcal{H}_{\mathbb{F}, k}(\beta) = \frac{(-8)^{k-1}}{\pi}\int\limits_{0}^{\infty}\frac{4\Xi_{\mathbb{F}}^{k}\left(\frac{t}{2}\right)}{\left(1 + t^2\right)^{k}} \cos\left(\frac{t}{2}\log\alpha\right)\mathrm{d}t.
\end{align}
\end{theorem}
For $k=1$, Theorem~\ref{Theorem Number field analogue Id 2nd paper} gives the following number-field extension of the Jacobi theta transformation.
\begin{corollary}\label{Case k=1 mt2}
Let $\mathtt{a}_{\mathbb{F}}(n)$ and $H_\mathbb{F}$ be defined as in \eqref{Dedekind zeta} and \eqref{Class number formula} respectively. Let $\alpha,\beta>0$ satisfy $\alpha\beta=1$. Then,
\begin{align*}
\mathcal{H}_{\mathbb{F}, 1}(\alpha) = \mathcal{H}_{\mathbb{F}, 1}(\beta) = \frac{4}{\pi}\int\limits_{0}^{\infty}\frac{\Xi_{\mathbb{F}}\left(\frac{t}{2}\right)}{\left(1 + t^2\right)} \cos\left(\frac{t}{2}\log\alpha\right)\mathrm{d}t,
\end{align*}
where
\begin{align*}
\mathcal{H}_{\mathbb{F}, 1}(x) &= \sqrt{x}\left\{\frac{\sqrt{D}H_\mathbb{F}}{(2\pi)^{r_2}x} - \sum_{n=1}^{\infty}\mathtt{a}_{\mathbb{F}}(n)\tilde{Z}_{r_1, r_2}\left(\frac{2^{r_2}\pi^{d/2} nx}{\sqrt{D}} \right)\right\}.
\end{align*}
\end{corollary}
For $\mathbb{F} = \mathbb{Q}$, the following corollary provides the extended form of the Jacobi theta relation \eqref{Jacobi relation} due to Hardy.
\begin{corollary}\label{F=Q, k=1}
When $\mathbb{F}=\mathbb{Q}$, Corollary \textup{\ref{Case k=1 mt2}} leads to the identity \eqref{Jacobi relation}.
\end{corollary}

For $k=2$, Theorem \ref{Theorem Number field analogue Id 2nd paper} yields the following extension of the Ramanujan--Koshliakov identity \eqref{Extended Koshliakov formula} to the number fields.
\begin{corollary}\label{Gen Koshliakov Formula}
Let $\mathbb{F}$ be a number field, $H_\mathbb{F}$ be the residue defined in \eqref{Class number formula}, and let $\gamma_{\mathbb{F}}$ denote the Euler-Kronecker constant of $\mathbb{F}$. Then, for any positive real numbers $\alpha$ and $\beta$ satisfying $\alpha\beta=1$, we have
\begin{align}
\mathcal{H}_{\mathbb{F}, 2}(\alpha) = \mathcal{H}_{\mathbb{F}, 2}(\beta) = -\frac{32}{\pi}\int\limits_{0}^{\infty}\frac{\Xi_{\mathbb{F}}^{2}\left(\frac{t}{2}\right)}{\left(1 + t^2\right)^{2}} \cos\left(\frac{t}{2}\log\alpha\right)\mathrm{d}t, \label{Extended number field RK formula}
\end{align}
where
\begin{align*}
\mathcal{H}_{\mathbb{F}, 2}(x) &= \frac{D H_{\mathbb F}^2}{(2\pi)^{2r_2}\sqrt{x}} \left[ 2\gamma_{\mathbb F} +\log\left(\frac{D}{\pi^d x}\right) -d\gamma -2(r_1+r_2)\log 2\right] \\
&- \sqrt{x}\sum_{n=1}^\infty \mathtt{a}_{\mathbb{F}, 2}(n)\tilde{Z}_{2r_1, 2r_2}\left(\frac{4^{r_2}\pi^{d} nx}{D} \right).
\end{align*}
\end{corollary}
\begin{corollary}\label{Gen Koshliakov Formula by Dixit}
If $\mathbb{F} = \mathbb{Q}$, then Corollary \textup{\ref{Gen Koshliakov Formula}} reduces to the extended Ramanujan--Koshliakov identity \eqref{Extended Koshliakov formula}.
\end{corollary}

\section{Prerequisites}
In this section, we record several standard results and lemmas that will be required in the proofs of our main theorems. We begin with two classical identities satisfied by the gamma function $\Gamma(s)$, namely Legendre's duplication formula and Euler's reflection formula, stated below,
\begin{align}
\Gamma(z)\Gamma\left(z + \frac{1}{2}\right)  &= 2^{1-2z}\sqrt{\pi}\Gamma(2z), \label{Legendre Duplication Formula} \\
\Gamma(z)\Gamma(1-z) &= \frac{\pi}{\sin(\pi z)}, \quad z \in \mathbb{C} \setminus \mathbb{Z}. \label{Reflection formula} 
\end{align}
These formulas play a central role in establishing the functional equation of the Riemann zeta function. Indeed, one can show that $\zeta(s)$ satisfies the asymmetric functional equation given by
\begin{align*}
\zeta(s) = 2^s \pi^{s-1} \sin\left(\frac{\pi s}{2}\right) \Gamma(1-s) \zeta(1-s).
\end{align*}
Motivated by this, we now turn to the number field setting and derive an analogous asymmetric functional equation for the Dedekind zeta function $\zeta_\mathbb{F}(s)$ in the following lemma.
\begin{lemma}\label{Series defn of Lambda_F}
Let $\mathbb{F}$ be a number field of degree $d = r_1 + 2r_2$. Then, $\zeta_\mathbb{F}(s)$ satisfies the following asymmetric form of the functional equation:
\begin{align}\label{Asymm fnal eqn for Dedekind zeta}
\zeta_\mathbb{F}(s) = \frac{\sqrt{D}}{2^{r_2}\pi^d}
\left(\frac{(2\pi)^d}{D} \right)^{s}\Gamma^{d}(1-s)\sin^{r_1}\left(\frac{\pi s}{2}\right)\sin^{r_2}(\pi s)\zeta_\mathbb{F}(1-s).
\end{align}
\end{lemma}
\begin{proof}
From \eqref{Dedekind zeta functional equation}, we have
\begin{align}\label{Func eqn 2.0}
\zeta_\mathbb{F}(s) = \left( \frac{D}{4^{r_2} \pi^d} \right)^{\frac{1}{2} - s}
\frac{\Gamma^{r_1}\left(\frac{1-s}{2}\right) \Gamma^{r_2}(1-s)}{\Gamma^{r_1}\left(\frac{s}{2}\right) \Gamma^{r_2}(s)}\zeta_\mathbb{F}(1-s).
\end{align}
We use the reflection formula \eqref{Reflection formula} in \eqref{Func eqn 2.0} to obtain
\begin{align*}
\zeta_\mathbb{F}(s) = \left( \frac{D}{4^{r_2} \pi^d} \right)^{\frac{1}{2} - s}\left[\frac{\Gamma\left(\frac{1-s}{2}\right)\Gamma\left(1-\frac{s}{2}\right)}{\pi}\right]^{r_1}\left[\frac{\Gamma^{2}(1-s)}{\pi}\right]^{r_2}\!\sin^{r_1}\left(\frac{\pi s}{2}\right)\sin^{r_2}(\pi s)\zeta_\mathbb{F}(1-s).
\end{align*}
We now employ \eqref{Legendre Duplication Formula} in the last equation to get
\begin{align*}
\zeta_\mathbb{F}(s) = \frac{\sqrt{D}}{2^{r_2}\pi^d}
\left(\frac{(2\pi)^d}{D} \right)^{s}\Gamma^{r_1}\left(1-s\right)\Gamma^{2r_2}(1-s)\sin^{r_1}\left(\frac{\pi s}{2}\right)\sin^{r_2}(\pi s)\zeta_\mathbb{F}(1-s),
\end{align*}
which gives the desired result.
\end{proof}
We now express the function $\lambda_{\mathbb{F}}(y)$, defined in \eqref{Generalized Lambda} in terms of a series expansion in the following lemma.
\begin{lemma}\label{Special integral lemma}
Let $\delta$ be a real number such that $0<\delta<1$ and $\mathbb{F}$ be a number field of degree $d$. Then, for $0<\Re(y)<1$, we have
\begin{align*}
\lambda_{\mathbb{F}}(y) &= \frac{1}{2\pi i}\int\limits_{(1+\delta)}
\frac{\zeta_\mathbb{F}(1-s)}{\sin^{d}(\pi s)}y^{-s} \mathrm{d}s = R^{(d+1)}_0(y) + R^{(r_2+1)}_1(y) + \sum_{k =1}^\infty R^{(d)}_{-k}(y),
\end{align*}
where the terms $R^{(d+1)}_0(y), R^{(r_2+1)}_1(y)$ and $R^{(d)}_{-k}(y)$ are given by
\begin{align*}
R^{(d+1)}_0(y) &= \frac{1}{d!}\lim_{s \to 0}\frac{\mathrm{d}^{d}}{\mathrm{d}s^{d}}\left(s^{d+1}\frac{\zeta_\mathbb{F}(1-s)}{\sin^{d}(\pi s)}y^{-s} \right),\\
R^{(r_2+1)}_1(y) &= \frac{1}{(r_2)!}\lim_{s\to 1}\frac{\mathrm{d}^{r_2}}{\mathrm{d}s^{r_2}}\left((s-1)^{r_2+1}\frac{\zeta_\mathbb{F}(1-s)}{\sin^{d}(\pi s)}y^{-s} \right), \\
R^{(d)}_{-k}(y) &= \frac{1}{(d-1)!}\lim_{s\to -k}\frac{\mathrm{d}^{d-1}}{\mathrm{d}s^{d-1}}\left((s+k)^{d}\frac{\zeta_\mathbb{F}(1-s)}{\sin^{d}(\pi s)}y^{-s} \right).
\end{align*}
\end{lemma}
\begin{proof}
We consider a rectangular contour with vertices $1+\delta - iT, 1+\delta + iT, \alpha +iT$ and $\alpha -iT$, where $\alpha < 0$. By Cauchy's residue theorem, we have
\begin{align*}
\frac{1}{2\pi i}\left( \int\limits_{1+\delta-iT}^{1+\delta +iT} + \int\limits_{1+\delta + iT}^{\alpha +iT} + \int\limits^{\alpha -iT}_{\alpha +iT} + \int\limits_{\alpha -iT}^{1+\delta -iT}\right) \frac{\zeta_\mathbb{F}(1-s)}{\sin^{d}(\pi s)}y^{-s} \mathrm{d}s &= \mathcal{R},
\end{align*}
where $\mathcal{R}$ is the sum of the residues of the integrand function inside the contour. As $T \to \infty$, the horizontal integrals vanish because the exponential decay of $\sin^{-d}(\pi s)$ dominates the polynomial growth of $\zeta_{\mathbb F}(1-s)$. Taking $\alpha \to -\infty$, the left vertical integral likewise vanishes, as $\sin^{-d}(\pi s)$ continues to suppress $\zeta_{\mathbb F}(1-s)$ while $y^{-s}$ decays to zero under the absolute convergence of $\zeta_{\mathbb F}(1-s)$. We now analyse the poles of the integrand function. The Dedekind zeta function $\zeta_{\mathbb F}(1-s)$ has a zero of order $r_1+r_2-1$ at $s=0$ and simple pole at $s=1$. Thus, the integrand has a pole of order $d+1$ at $s=0$, a pole of order $r_2+1$ at $s=1$, and poles of order $d$ at the negative integers. Therefore,
\begin{align*}
\frac{1}{2\pi i}\int\limits_{(1+\delta)}\frac{\zeta_\mathbb{F}(1-s)}{\sin^{d}(\pi s)}y^{-s} \text{d}s &= R^{(d+1)}_0(y) + R^{(r_2+1)}_1(y) + \sum_{k =1}^\infty R^{(d)}_{-k}(y),
\end{align*}
where the terms $R^{(d+1)}_0(y), R^{(r_2+1)}_1(y)$ and $R^{(d)}_{-k}(y)$ are given by
\begin{align*}
R^{(d+1)}_0(y) &= \frac{1}{d!}\lim_{s \to 0}\frac{\text{d}^{d}}{\text{d}s^{d}}\left(s^{d+1}\frac{\zeta_\mathbb{F}(1-s)}{\sin^{d}(\pi s)}y^{-s} \right),\\
R^{(r_2+1)}_1(y) &= \frac{1}{(r_2)!}\lim_{s\to 1}\frac{\mathrm{d}^{r_2}}{\mathrm{d}s^{r_2}}\left((s-1)^{r_2+1}\frac{\zeta_\mathbb{F}(1-s)}{\sin^{d}(\pi s)}y^{-s} \right), \\
R^{(d)}_{-k}(y) &= \frac{1}{(d-1)!}\lim_{s\to -k}\frac{\text{d}^{d-1}}{\text{d}s^{d-1}}\left((s+k)^{d}\frac{\zeta_\mathbb{F}(1-s)}{\sin^{d}(\pi s)}y^{-s} \right).
\end{align*}
\end{proof}
\begin{remark}
Upon taking $\mathbb{F}=\mathbb{Q}$ in the above lemma, we obtain
\begin{align}
\lambda_{\mathbb{Q}}(y) = \frac{1}{2\pi i}\int\limits_{(1+\delta)}
\frac{\zeta(1-s)}{\sin(\pi s)}y^{-s} \mathrm{d}s &= R^{(2)}_0(y) + R_1(y) + \sum_{k =1}^\infty R_{-k}(y), \label{Integral from Titchmarsh unsolved}
\end{align}
where the terms on the right-hand side of \eqref{Integral from Titchmarsh unsolved}
are given by
\begin{align*}
R^{(2)}_0(y) &= \lim_{s \to 0}\frac{\mathrm{d}}{\mathrm{d}s}\left(s^{2}\frac{\zeta(1-s)}{\sin(\pi s)}y^{-s} \right) = \frac{1}{\pi}(\gamma+\log(y)), \\
R_1(y) &= \lim_{s \rightarrow 1} (s-1)\frac{\zeta(1-s)}{\sin(\pi s)}y^{-s} = \frac{1}{2\pi y}, \\
R_{-k}(y) &= \lim_{s \rightarrow -k} (s+n)\frac{\zeta(1-s)}{\sin(\pi s)}y^{-s} = \frac{(-1)^k\zeta(1+k)y^k}{\pi},
\end{align*}
where we have used the series expansions of $\zeta(s)$, $y^{-s}$ and $\frac{1}{\sin(\pi s)}$ to evaluate $R^{(2)}_0(y)$. Substituting the above values of $R^{(2)}_0(y), R_1(y)$ and $R_{-k}(y)$ in \eqref{Integral from Titchmarsh unsolved}, we recover the following classical result from \cite[p.~201, formula~5.74]{Oberhettinger}:
{\allowdisplaybreaks
\begin{align}
\lambda_{\mathbb{Q}}(y) = \frac{1}{2\pi i}\int\limits_{(1+\delta)}
\frac{\zeta(1-s)}{\sin(\pi s)}y^{-s} \mathrm{d}s &= \frac{1}{\pi}(\gamma +\log(y)) + \frac{1}{2\pi y} + \sum_{k=1}^\infty \frac{(-1)^k\zeta(1+k)y^k}{\pi} \notag \\
&= -\frac{1}{\pi}\left(\psi(y+1)- \frac{1}{2y} - \log(y)\right)  \notag \\
&= -\frac{1}{\pi}\left(\psi(y) + \frac{1}{2y} - \log(y)\right) = -\frac{\lambda(y)}{\pi}, \label{Integral from Titchmarsh}
\end{align}}
where $\lambda(y)$ is defined in \eqref{Series Lambda}, and in the last step, we have utilized $\psi(x+1) = \psi(x) + \frac{1}{x}$ \cite[p.~54]{Temme} and the Taylor series expansion $$\psi(x+1) = -\gamma - \sum_{n=1}^{\infty} (-1)^k\zeta(1+k)x^k, \quad |x|<1.$$
\end{remark}
To keep this paper self-contained, we give a number-field analogue of Titchmarsh's formula \cite[p.~35]{Titchmarsh}.
Let $f(t)$ be a function such that $f(t) := |\phi(it)|^2 = \phi(it)\phi(-it), \,\, t \in \mathbb{R}$, where $\phi(s)$ is an analytic function with the characterisation $\overline{\phi(s)} = \phi(\bar{s})$.
For $z \in \mathbb{C}$, we have
\begin{align*}
\int\limits_0^\infty f(t)\Xi_{\mathbb{F}}(t)\cos(zt) \text{d}t = \frac{1}{2}\int\limits_{-\infty}^\infty f(t)\Xi_{\mathbb{F}}(t)y^{it}\text{d}t 
= \frac{1}{2}\int\limits_{-\infty}^\infty \phi(it)\phi(-it)\xi_{\mathbb{F}}\left(\frac{1}{2} + it\right) y^{it}\text{d}t,
\end{align*}
where $\xi_{\mathbb{F}}(s)$ and $\Xi_{\mathbb{F}}(t)$ are defined in \eqref{xi function} and \eqref{Xi_F(t)}. We let $\frac{1}{2}+it = s$, replace $t \mapsto \frac{t}{2}$ and $z \mapsto 2z$ to obtain
\begin{align}\label{Special int}
\int\limits_0^\infty f\left(\frac{t}{2}\right)\Xi_{\mathbb{F}}\left(\frac{t}{2}\right)\cos(zt) \text{d}t &= \frac{1}{i\sqrt{y}}\int\limits_{\left(\frac{1}{2}\right)}\phi\left(s - \frac{1}{2}\right)\phi\left(\frac{1}{2} - s\right)\xi_{\mathbb{F}}(s)y^s \text{d}s,
\end{align}
where $y = e^{2z}$. The above formula will be used in the proof of our second main theorem.
We shall also use the following Laurent and Taylor expansions. As $s\to1$, the Dedekind zeta function has the Laurent expansion
\begin{align}
(s-1)\zeta_{\mathbb{F}}(s) &= H_{\mathbb{F}}+\gamma_{\mathbb{F}}(s-1)+O\!\left((s-1)^2\right),
\label{zeta around 1}
\end{align}
where $\gamma_{\mathbb{F}}$ denotes the Euler-Kronecker constant of the number field $\mathbb{F}$ and $H_{\mathbb F}$ is the residue of $\zeta_{\mathbb F}(s)$ at pole $s=1$, as given by the class number formula \eqref{Class number formula}. As $s\to1$, the gamma function and its powers have the expansions
\begin{align}
\Gamma(s) &= 1 - \gamma(s-1) + O((s-1)^2), \notag \\
\Rightarrow \Gamma^k(s) &= 1 - k\gamma(s-1) + O((s-1)^2), \label{Series Gamma(s)} \\
\Gamma^k\left(\frac{s}{2}\right) &= \pi^\frac{k}{2} - k\pi^\frac{k}{2}\left(\frac{\gamma}{2} + \log2\right)(s-1) + O((s-1)^2). \label{Series Gamma(s/2)}
\end{align}
We also have
\begin{align}\label{exp series}
x^\frac{s}{2} &= \sqrt{x} + \frac{\sqrt{x}}{2}\log(x)(s-1) + O((s-1)^2).
\end{align}
Combining \eqref{zeta around 1}, \eqref{Series Gamma(s)}, \eqref{Series Gamma(s/2)}, and \eqref{exp series}, we obtain the Laurent expansion of $\Omega_{\mathbb F}(s)$ at $s=1$: 
\begin{align}
&\Omega_{\mathbb{F}}(s) = \left(\frac{D}{4^{r_2}\pi^d}\right)^{\frac{s}{2}}\Gamma^{r_1}\left(\frac{s}{2}\right)\Gamma^{r_2}(s)\zeta_{\mathbb{F}}(s) \notag \\
&=\frac{\sqrt{D}}{(2\pi)^{r_2}}
\frac{H_{\mathbb F}}{s-1} + \frac{\sqrt{D}H_{\mathbb F}}{(2\pi)^{r_2}}
\left[\gamma_{\mathbb F} + \frac{\log\left(\frac{D}{\pi^d}\right)}{2}-\frac{d\gamma}{2}-(r_1+r_2)\log 2\right] + O(s-1). \label{Laurent series for Omega at s=1}
\end{align}

\section{Proofs of main results}

\begin{proof}[Theorem \rm{\ref{Theorem Number field analogue Id 220}}][]
Let us define
\begin{align*}
f(t) := \frac{\Xi_{\mathbb{F}}(t)}{\left(\frac{1}{4} + t^2\right)^{2-r_1-r_2}}
\left[\Gamma\!\left(\frac{it}{2}-\frac{1}{4} \right)\Gamma\!\left(-\frac{it}{2}-\frac{1}{4}\right)\right]^{r_1}\!\left[\Gamma\!\left(it-\frac{1}{2}\right)\Gamma\!\left(-it -\frac{1}{2}\right)\right]^{r_2}.
\end{align*}
Considering the function $\phi(s)$ as
\begin{align*}
\phi(s) = \left(\frac{1}{2}+s\right)^{r_1+r_2-2}\sqrt{\xi_{\mathbb{F}}\!\left(\frac{1}{2}-s\right)}\Gamma^{r_1}\!\left(\frac{s}{2}-\frac{1}{4}\right)\Gamma^{r_2}\!\left(s-\frac{1}{2} \right),
\end{align*}
we have $f(t) = \phi(it)\phi(-it)$.
Substituting this into \eqref{Special int} with $y=e^{2z}$ and applying the functional equation $s\Gamma(s) = \Gamma(s+1)$, we obtain
{\allowdisplaybreaks
\begin{align}
&\int\limits_{0}^{\infty}\frac{16\, \Xi_\mathbb{F}^2\left(\frac{t}{2}\right)}{(1+t^2)^{2-r_1-r_2}} \left[\frac{\Gamma\!\left(\frac{it}{4} - \frac{1}{4}\right)\Gamma\!\left(- \frac{it}{4}-\frac{1}{4} \right)}{4}\right]^{r_1}\left[\frac{\Gamma\!\left(\frac{it}{2} - \frac{1}{2}\right)\Gamma\!\left(-\frac{it}{2} - \frac{1}{2}  \right)}{4}\right]^{r_2} \cos(zt)\text{d}t \notag\\
&= \frac{1}{i\sqrt{y}}
\int\limits_{\left(\frac{1}{2}\right)}\frac{(\xi_\mathbb{F}(s))^2}{(-s(s-1))^{2-r_1-r_2}}\left[\Gamma\left(\frac{s-1}{2}\right)\Gamma\left(-\frac{s}{2}\right)\right]^{r_1}\left[\Gamma(-s)\Gamma(s-1)\right]^{r_2} y^s \text{d}s \notag\\
&= \frac{4^{r_1}}{i\sqrt{y}}
\int\limits_{\left(\frac{1}{2}\right)}
\frac{(\xi_\mathbb{F}(s))^2}{(s(s-1))^{2}}
\left[\Gamma\left(\frac{s+1}{2}\right)\Gamma\left(1-\frac{s}{2}\right)\right]^{r_1}\left[\Gamma(1-s)\Gamma(s)\right]^{r_2}
y^s \text{d}s \notag\\
&= \frac{4^{r_1-1}}{i\sqrt{y}}
\int\limits_{\left(\frac{1}{2}\right)}\Bigl[\Gamma\!\left(\frac{s+1}{2}\right)\Gamma\!\left(1-\frac{s}{2}\right)\Gamma^2\!\left(\frac{s}{2}\right)\Bigr]^{r_1}\Bigl[\Gamma(1-s)\Gamma^3(s)\Bigr]^{r_2} 
\zeta_\mathbb{F}^2(s)\left(\frac{yD}{4^{r_2}\pi^d}\right)^s \text{d}s \notag \\
&= \frac{\sqrt{D}\pi^{\frac{3r_1}{2}+r_2}2^{3r_1-r_2}}{4i\sqrt{y}}
\int\limits_{\left(\frac{1}{2}\right)}
\frac{\zeta_\mathbb{F}(s)\zeta_\mathbb{F}(1-s)}{\sin^{d}(\pi s)}y^s \text{d}s, \label{1st integral}
\end{align}}
where in the penultimate step, we invoked \eqref{xi function}, followed by Legendre’s duplication formula \eqref{Legendre Duplication Formula}, Euler’s reflection formula \eqref{Reflection formula} for $\Gamma(s)$, and the asymmetric functional equation \eqref{Asymm fnal eqn for Dedekind zeta} for $\zeta_\mathbb{F}(s)$.
We now denote the integrand in \eqref{1st integral} by
\begin{align*}
I_{\mathbb{F}, y}(s) 
&= \frac{\zeta_\mathbb{F}(s)\zeta_\mathbb{F}(1-s)}{\sin^{d}(\pi s)}y^s.
\end{align*}
To evaluate the integral in \eqref{1st integral}, we shift the line of integration from $\Re(s)=\frac{1}{2}$ to $\Re(s)=1+\delta$, with $\delta\in(0,1)$. This shift introduces a pole of order $r_2+2$ at $s=1$, which must be treated carefully.
Given $T>0$, the residue theorem implies
\begin{align}
\int\limits_{\frac{1}{2}-iT}^{\frac{1}{2}+iT} I_{\mathbb{F}, y}(s) \text{d}s &= \Bigg(
\int\limits_{\frac{1}{2}-iT}^{1+\delta-iT}
+ \int\limits_{1+\delta-iT}^{1+\delta+iT}
+ \int\limits_{1+\delta+iT}^{\frac{1}{2}+iT}
\Bigg)\! I_{\mathbb{F}, y}(s) \text{d}s - 2\pi i\, \mathcal{R}^{(r_2+2)}_1(y),
\label{CRT 1}
\end{align}
where
\begin{align*}
\mathcal{R}^{r_2+2}_1(y) &= \frac{1}{(r_2+1)!}\lim_{s\to 1}\frac{\text{d}^{r_2+1}}{\text{d}s^{r_2+1}}\left((s-1)^{r_2+2}I_{\mathbb{F}, y}(s)\right).
\end{align*}
From the Stirling formula of gamma function, we deduce that the integrals along the horizontal lines from $\frac{1}{2}\pm iT$ to $1\pm \delta+iT$ vanish as $T \rightarrow \infty$.
Now we further simplify the vertical integral from $1+\delta-i\infty$ to $1+\delta+i\infty$ using the series definition \eqref{Dedekind zeta} of $\zeta_\mathbb{F}(s)$ and interchange the integral with the summation due the absolute convergence to have
\begin{align}\label{line integral 1 to delta}
\int\limits_{(1+\delta)} I_{\mathbb{F},y}(s)\text{d}s &= \sum_{m=1}^{\infty}\mathtt{a}_{\mathbb{F}}(m)\lambda_\mathbb{F}\left(\frac{m}{y}\right),
\end{align}
where
\begin{align*}
\lambda_\mathbb{F}\left(\frac{m}{y}\right) &= \int\limits_{(1+\delta)}\frac{\zeta_\mathbb{F}(1-s)}{\sin^{d}(\pi s)}\left(\frac{m}{y}\right)^{-s} \text{d}s.
\end{align*}
Hence substituting the value of \eqref{line integral 1 to delta} in \eqref{CRT 1} with $T \to \infty$ and using the final value of \eqref{CRT 1} in \eqref{1st integral}, we get
\begin{align}\label{The last integral}
&-\frac{2^{6-5r_1-r_2}}{\sqrt{D}\pi^{\frac{3r_1}{2}+r_2}i}\int\limits_{0}^{\infty}\left|\Xi_\mathbb{F}\left(\frac{t}{2}\right)\Gamma^{r_1}\left(\frac{-1+it}{4}\right)\Gamma^{r_2}\left(\frac{-1+it}{2}\right)\right|^2\frac{\cos(zt)}{(1+t^2)^{2-r_1-r_2}}\,\text{d}t \notag \\
&= \frac{2\pi i}{\sqrt{y}}\sum_{m=1}^{\infty}\mathtt{a}_{\mathbb{F}}(m)\lambda_\mathbb{F}\left(\frac{m}{y}\right) - \frac{2\pi i}{\sqrt{y}}\, \mathcal{R}^{(r_2+2)}_1(y),
\end{align}
where
\begin{align*}
\mathcal{R}^{(r_2+2)}_1(y) &= \frac{1}{(r_2+1)!}\lim_{s\to 1}\frac{\text{d}^{r_2+1}}{\text{d}s^{r_2+1}}\left((s-1)^{r_2+2}I_{\mathbb{F}, y}(s)\right).
\end{align*}
In \eqref{The last integral}, we have used the fact that $\overline{\Gamma(s)} = \Gamma(\overline{s})$ and $\Xi_{\mathbb F}(t)$ is real for real values of $t$. Substituting $z=\frac{1}{2}\log\alpha$ into \eqref{The last integral} and observing that $y=e^{2z}=\alpha$ with $\alpha\beta=1$, we get
\begin{align}
& \frac{2^{5-5r_1-r_2}}{\sqrt{D}\pi^{\frac{3r_1}{2}+r_2+1}}\int\limits_{0}^{\infty}\left|\Xi_\mathbb{F}\left(\frac{t}{2}\right)\Gamma^{r_1}\left(\frac{-1+it}{4}\right)\Gamma^{r_2}\left(\frac{-1+it}{2}\right)\right|^2\frac{\cos\left(\frac{t}{2}\log\alpha\right)}{(1+t^2)^{2-r_1-r_2}}\,\text{d}t \notag \\
&= \sqrt{\beta}\left\{\sum_{m=1}^{\infty}\mathtt{a}_{\mathbb{F}}(m)\lambda_\mathbb{F}(m\beta) - \mathcal{R}^{(r_2+2)}_1\left(\frac{1}{\beta}\right)\right\}. \label{Eqn with beta}
\end{align}
Since the left hand side of the above equation is unchanged under the substitution $\alpha \mapsto \beta$, we interchange $\beta$ by $\alpha$ in \eqref{Eqn with beta} and combine it with \eqref{Eqn with beta} to get the desired transformation formula \eqref{Number field analogue of transformation formula}.
\end{proof}

\begin{proof}[Corollary \rm{\ref{Recovery of main id}}][]
Taking $\mathbb{F} = \mathbb{Q}$ in Theorem \ref{Theorem Number field analogue Id 220}, and using \eqref{Integral from Titchmarsh} together with the fact that $\mathtt{a}_{\mathbb{Q}}(m) = 1$, we obtain
\begin{align}\label{Lambda}
\sum_{m=1}^{\infty}\mathtt{a}_{\mathbb{Q}}(m)\lambda_\mathbb{Q}(mx) = -\frac{1}{\pi}\sum_{m=1}^{\infty}\lambda(mx).
\end{align}
Now, we calculate the term $\mathcal{R}^{(r_2+2)}_1\left(\frac{1}{x}\right) = \mathcal{R}^{(2)}_1\left(\frac{1}{x}\right)$ in \eqref{G(x) for no field}, where
\begin{align}
\mathcal{R}^{(2)}_1(x) &= \lim_{s\to 1}\frac{\text{d}}{\text{d}s}\left((s-1)^{2}\frac{\zeta(s)\zeta(1-s)}{\sin(\pi s)}x^s\right).
\end{align}
We use the following Laurent series expansions to find the above residual term.
\begin{align*}
\frac{(s-1)}{\sin(\pi s)} &= -\frac{1}{\pi} - \frac{\pi}{6}(s-1)^2 + \ldots, \\
(s-1)\zeta(s) &= 1 + \gamma(s-1) + \ldots, \\
x^s &= x + x\log(x)(s-1) + \ldots, \\
\zeta(1-s) &= -\frac{1}{2} + \frac{1}{2}\log(2\pi)(s-1) + \ldots.
\end{align*}
Hence, collecting the coefficient of $(s-1)$ in the Laurent series expansion of $(s-1)^{2}\frac{\zeta(s)\zeta(1-s)}{\sin(\pi s)}x^s$, we have
\begin{align}
\mathcal{R}^{(2)}_1(x) &= \frac{x}{2\pi}\left[\gamma - \log\left(\frac{2\pi}{x}\right)\right], \notag \\
\Rightarrow \mathcal{R}^{(2)}_1\left(\frac{1}{x}\right) &=  \frac{1}{2\pi x}\left[\gamma - \log(2\pi x)\right].\label{Res for Q}
\end{align}
Hence, substituting the value of \eqref{Lambda} and \eqref{Res for Q} in Theorem \ref{Theorem Number field analogue Id 220}, we get the desired identity \eqref{Id 220}.
\end{proof}

\begin{proof}[Theorem \rm{\ref{Theorem Number field analogue Id 2nd paper}}][]
Let us define 
\begin{align*}
f(t) := \frac{2^k\Xi_{\mathbb{F}}^{k-1}(t)}{\left(\frac{1}{4} + t^2\right)^{k}}.
\end{align*}
We know that $f(t) = \phi(it)\phi(-it)$, so we have
\begin{align*}
\phi(s) = \frac{2^{\frac{k}{2}}\xi_{\mathbb{F}}^{\frac{k-1}{2}}\left(\frac{1}{2}-s\right)}{\left(\frac{1}{2} + s\right)^{\frac{k}{2}}\left(\frac{1}{2} - s\right)^{\frac{k}{2}}},
\end{align*}
where $\xi_{\mathbb{F}}(s)$ is defined in \eqref{xi function}. Substituting the above value of $\phi(s)$ in \eqref{Special int} with $y=e^{2z}$, we obtain
{\allowdisplaybreaks
\begin{align}
\int\limits_{0}^{\infty}\frac{2^{3k}\Xi_{\mathbb{F}}^{k}\left(\frac{t}{2}\right)}{\left(1 + t^2\right)^{k}} \cos(zt)\text{d}t 
&= \frac{1}{i\sqrt{y}}
\int\limits_{\left(\frac{1}{2}\right)}\frac{2^k\left[\xi_\mathbb{F}(1-s)\xi_\mathbb{F}(s)\right]^{\frac{k-1}{2}}}{s^k(1-s)^k} \xi_\mathbb{F}(s) y^s \text{d}s \notag\\
&= \frac{1}{i\sqrt{y}}
\int\limits_{\left(\frac{1}{2}\right)}\left(\frac{2\xi_\mathbb{F}(s)}{s(1-s)}\right)^k y^s \text{d}s \notag\\
&= \frac{(-1)^k}{i\sqrt{y}}
\int\limits_{\left(\frac{1}{2}\right)} \Omega_\mathbb{F}^k(s)y^s \text{d}s,
\label{1st integral I1}
\end{align}}
where in the penultimate step, we invoked \eqref{xi function}. To solve the above integral in \eqref{1st integral I1}, we shift the line of integration from $\Re(s)=\frac{1}{2}$ to $\Re(s)=1+\delta$, with $\delta\in(0,1)$. This shift introduces a pole of order $k$ at $s=1$, which must be treated carefully. Let $T>0$, the residue theorem implies 
\begin{align*}
\int\limits_{\frac{1}{2}-iT}^{\frac{1}{2}+iT} \Omega_\mathbb{F}^k(s)y^s \text{d}s &= \Bigg(
\int\limits_{\frac{1}{2}-iT}^{1+\delta-iT}
+ \int\limits_{1+\delta-iT}^{1+\delta+iT}
+ \int\limits_{1+\delta+iT}^{\frac{1}{2}+iT}
\Bigg)\! \Omega_\mathbb{F}^k(s)y^s \text{d}s - 2\pi i\, R^{(k)}_1(y),
\end{align*}
where
\begin{align*}
R^{(k)}_1(y) &= \frac{1}{(k-1)!}\lim_{s\to 1}\frac{\text{d}^{k-1}}{\text{d}s^{k-1}}\left((s-1)^{k}\Omega_\mathbb{F}^k(s)y^s \right).
\end{align*}
The horizontal integrals vanish as $T \rightarrow \infty$ due to the Stirling formula for the gamma factors in the integrand and we are left with
\begin{align}\label{After T infinity}
\int\limits_{\left(\frac{1}{2}\right)}\Omega_\mathbb{F}^k(s)y^s \text{d}s = \int\limits_{\left(1+\delta\right)}\Omega_\mathbb{F}^k(s)y^s \text{d}s - 2\pi i\, R^{(k)}_1(y).
\end{align}
Note that the integral on the right hand side has been evaluated by the author and Maji \cite[p.~14]{BM2}, and can be written as
\begin{align}\label{Half integral value}
\int\limits_{\left(1+\delta\right)}\Omega_\mathbb{F}^k(s)y^s \text{d}s = 2\pi i\, S_{\mathbb{F}, k}\left(\frac{1}{y^2}\right),
\end{align}
where $S_{\mathbb{F}, k}(x)$ is given by
\begin{align*}
S_{\mathbb{F}, k}(x) := \sum_{n=1}^{\infty}\mathtt{a}_{\mathbb{F}, k}(n)\tilde{Z}_{kr_1, kr_2}\left(\frac{2^{kr_2}\pi^{kd/2} n\sqrt{x}}{\sqrt{D^k}} \right), \quad |\Arg(x)| < \frac{\pi d}{2}.
\end{align*}
Substituting \eqref{Half integral value} in \eqref{After T infinity} and then using the final value of the integral on the left hand side of \eqref{After T infinity} into \eqref{1st integral I1}, we get
\begin{align*}
\int\limits_{0}^{\infty}\frac{2^{3k}\Xi_{\mathbb{F}}^{k}\left(\frac{t}{2}\right)}{\left(1 + t^2\right)^{k}} \cos(zt)\text{d}t &= \frac{(-1)^k 2\pi}{\sqrt{y}}\left(S_{\mathbb{F}, k}\left(\frac{1}{y^2}\right) - R^{(k)}_1(y)\right).
\end{align*}
Substituting $z=\frac{1}{2}\log\alpha$ into the last equation and observing that $y=e^{2z}=\alpha$ with $\alpha\beta=1$, we get
\begin{align}\label{One side with beta}
\frac{(-1)^{k-1}}{2\pi}\int\limits_{0}^{\infty}\frac{2^{3k}\Xi_{\mathbb{F}}^{k}\left(\frac{t}{2}\right)}{\left(1 + t^2\right)^{k}} \cos\left(\frac{t}{2}\log\alpha\right)\text{d}t &= \sqrt{\beta}\left(R^{(k)}_1\left(\frac{1}{\beta}\right) - S_{\mathbb{F}, k}\left(\beta^2\right)\right).
\end{align}
Since the left-hand side of the above equation is invariant under the interchange $\alpha\leftrightarrow\beta$, we replace $\beta$ by $\alpha$ in \eqref{One side with beta} and combine the resulting relation with \eqref{One side with beta} to obtain
\begin{align*}
\frac{(-1)^{k-1}}{2\pi}\int\limits_{0}^{\infty}\frac{2^{3k}\Xi_{\mathbb{F}}^{k}\left(\frac{t}{2}\right)}{\left(1 + t^2\right)^{k}} \cos\left(\frac{t}{2}\log\alpha\right)\text{d}t &= \sqrt{\alpha}\left(R^{(k)}_1\left(\frac{1}{\alpha}\right) - S_{\mathbb{F}, k}\left(\alpha^2\right)\right) \\
&= \sqrt{\beta}\left(R^{(k)}_1\left(\frac{1}{\beta}\right) - S_{\mathbb{F}, k}\left(\beta^2\right)\right),
\end{align*}
which is the desired identity \eqref{Extended trans formula}.
\end{proof}

\begin{proof}[Corollary \rm{\ref{Case k=1 mt2}}][]
For $k=1$, the term $R^{(k)}_1(x)$ in Theorem \ref{Theorem Number field analogue Id 2nd paper} reduces to
\begin{align*}
R^{(1)}_1(x) &= \lim_{s\to 1}(s-1)\Omega_\mathbb{F}(s)x^{s} \\
&= \left(\frac{D}{4^{r_2}\pi^d}\right)^{\frac{1}{2}}\Gamma^{r_1}\left(\frac{1}{2}\right)x\lim_{s\to 1}(s-1)\zeta_{\mathbb{F}}(s) \\
&= \frac{\sqrt{D}x}{(2\pi)^{r_2}}H_\mathbb{F}.
\end{align*}
Substituting this expression for $R^{(1)}_1(x)$ with $k=1$ into Theorem \ref{Theorem Number field analogue Id 2nd paper} yields the desired identity in Corollary \ref{Case k=1 mt2}.
\end{proof}

\begin{proof}[Corollary \rm{\ref{F=Q, k=1}}][]
When $\mathbb{F} = \mathbb{Q}$, we have $r_1=1, r_2=0, D=1, d=1, \mathtt{a}_{\mathbb{F}}(n) = 1$ and $H_\mathbb{F} = 1$. We also use \eqref{Z value real} and then \eqref{Exp fn formula} to obtain
\begin{align*}
\tilde{Z}_{r_1, r_2}\left(\frac{2^{r_2}\pi^{d/2} nx}{\sqrt{D}} \right) = 2V(\pi n^2 x^2|0) = 2e^{- n^2\pi x^2}
\end{align*}
Using these values in Corollary \ref{Case k=1 mt2}, we prove Corollary \ref{F=Q, k=1}.
\end{proof}

\begin{proof}[Corollary \rm{\ref{Gen Koshliakov Formula}}][]
For $k=2$, the term $R^{(k)}_1(x)$ in Theorem \ref{Theorem Number field analogue Id 2nd paper} reduces to
\begin{align*}
R^{(2)}_1(x) &= \lim_{s\to 1}\frac{\mathrm{d}}{\mathrm{d}s}\left((s-1)^2\Omega_\mathbb{F}^2(s)x^{s} \right).
\end{align*}
To solve the above term, we use \eqref{exp series} and \eqref{Laurent series for Omega at s=1} to find the Laurent series expansion of $\Omega_\mathbb{F}^2(s)x^s$ around 1 to have
\begin{align}
&(s-1)^2\Omega_{\mathbb F}^2(s)x^s \notag \\
&= \frac{D H_{\mathbb F}^2 x}{(2\pi)^{2r_2}}
\bigg[ 1+\bigg\{ 2\gamma_{\mathbb F} +\log\left(\frac{Dx}{\pi^d}\right) -d\gamma -2(r_1+r_2)\log 2 \bigg\}(s-1) +O\left((s-1)^2\right) \bigg].\notag
\end{align}
In particular,
\begin{align*}
R^{(2)}_1(x) &= \frac{D H_{\mathbb F}^2x}{(2\pi)^{2r_2}} \left[ 2\gamma_{\mathbb F} +\log\left(\frac{Dx}{\pi^d}\right) -d\gamma -2(r_1+r_2)\log 2\right].
\end{align*}
Upon inserting the above expression for $R^{(2)}_1(x)$ into \eqref{Extended trans formula} and taking $k=2$ in Theorem \ref{Theorem Number field analogue Id 2nd paper}, we obtain Corollary \ref{Gen Koshliakov Formula}.
\end{proof}

\begin{proof}[Corollary \rm{\ref{Gen Koshliakov Formula by Dixit}}][]
For $\mathbb{F}=\mathbb{Q}$, we have
$$
r_1=1,\quad r_2=0,\quad D=1,\quad d=1,\quad
\gamma_{\mathbb{F}}=\gamma,\quad
\mathtt{a}_{\mathbb{F},2}(n)=d(n),\quad H_{\mathbb{F}}=1.
$$
Moreover, by \eqref{Z value real} and \eqref{Bessel fn formula},
$$
\tilde{Z}_{2r_1,2r_2}\left(\frac{4^{r_2}\pi^d nx}{D}\right)
=\tilde{Z}_{2,0}(\pi nx)
=2V(\pi^2n^2x^2\mid0,0)
=4K_0(2\pi nx).
$$
Consequently, the function $\mathcal{H}_{\mathbb{F},2}(x)$ in Corollary \ref{Gen Koshliakov Formula} becomes
\begin{align*}
\mathcal{H}_{\mathbb{Q},2}(x)
&= \frac{1}{\sqrt{x}}\left[2\gamma+\log\left(\frac{1}{\pi x}\right)-\gamma-2\log2\right] - 4\sqrt{x}\sum_{n=1}^\infty d(n)K_0(2\pi nx)\\
&=\frac{\gamma -\log(4\pi x)}{\sqrt{x}} - 4\sqrt{x}\sum_{n=1}^\infty d(n)K_0(2\pi nx).
\end{align*}
Substituting the above expression for $\mathcal{H}_{\mathbb{Q},2}(x)$ into \eqref{Extended number field RK formula}, we get the identity \eqref{Extended Koshliakov formula}.
\end{proof}

\section{Concluding remarks}
In this paper, we established a number field analogue of Ramanujan's identity \eqref{Id 220} and obtained an integral representation for the transformation formula established by the present author and Maji \cite{BM2}. As a special case, corresponding to $k=1$, Corollary \ref{Case k=1 mt2} yields a number field analogue of the extended Jacobi theta relation, together with its corresponding integral representation. The Jacobi theta relation played a fundamental role in Hardy's proof of the infinitude of non-trivial zeros of the Riemann zeta function on the critical line. This naturally raises the question of whether the integral representation in Corollary \ref{Case k=1 mt2} can be exploited to investigate the non-trivial zeros of the Dedekind zeta function on the critical line. In particular, it would be interesting to determine whether an analogue of Hardy's argument can be developed in the number field setting.

\textbf{Acknowledgements}
The author extends heartfelt thanks to Prof. Atul Dixit for his time, insightful discussion and generous guidance. The author is deeply grateful to her supervisor, Dr. Bibekananda Maji, for carefully reviewing the manuscript and offering valuable feedback that greatly enhanced the quality of this work. Financial support from the Prime Minister’s Research Fellowship (PMRF), Government of India, under Grant No. 2102227, is gratefully acknowledged. The author also wishes to thank IIT Indore for fostering an encouraging and collaborative research environment.

\textbf{Data Availability Statement}
No new data were generated or analyzed in this work.

\end{document}